\documentclass[11pt,reqno]{amsart}
\usepackage[T1]{fontenc}
\usepackage{graphicx}
\usepackage{cite}

\usepackage{color}
\definecolor{MyLinkColor}{rgb}{0,0,0.4}
\newcommand{\e}{\varepsilon}

\newcommand{\p}{\partial}

\newcommand{\ov}{\overline}

\newcommand{\cE}{\mathcal{E}}

\newcommand{\cK}{{\mathcal K}}

\newcommand{\R}{\mathbb{R}}

\newcommand{\N}{\mathbb{N}}

\DeclareMathOperator{\supp}{supp}
\newtheorem{thm}{Theorem}[section]

\newtheorem{lemma}[thm]{Lemma}
\newtheorem{cor}[thm]{Corollary}
\numberwithin{equation}{section}   
\title[A singular limit of the thin film Muskat problem]{The porous medium equation as a singular limit of the thin film Muskat problem}
\thanks{Partially supported by DFG Research Training Group~2339 ``Interfaces, Complex Structures, and Singular Limits in Continuum Mechanics - Analysis and Numerics''}
\author[Ph. Lauren\c cot]{Philippe Lauren\c cot}
\address{Institut  de Math\'ematiques de Toulouse, UMR~5219, Universit\'e de Toulouse, CNRS, F-31062 Toulouse cedex~9, France}
\email{laurenco@math.univ-toulouse.fr}

 \author[B.--V. Matioc]{Bogdan--Vasile Matioc}
\address{Fakult\"at f\"ur Mathematik, Universit\"at Regensburg,   93040 Regensburg, Deutschland.}
 \email{bogdan.matioc@ur.de}


\usepackage[colorlinks=true,linkcolor=MyLinkColor,citecolor=MyLinkColor]{hyperref} 

\begin{document}

\begin{abstract}
The singular limit of the thin film Muskat problem is performed when the density (and possibly the viscosity) of the lighter fluid vanishes and the porous medium equation is identified as the limit problem. In particular, the height of the denser fluid is shown to converge towards the solution to the porous medium equation and an explicit rate for this convergence is provided in space dimension $d\leq 4$. Moreover, the limit of the height of the lighter fluid is determined in a certain regime and is given by the corresponding initial condition.
\end{abstract}

\maketitle

\section{Introduction and  main results}\label{Sec:1}

The thin film Muskat problem is the strongly coupled degenerate parabolic system
\begin{equation}\label{P}
\left\{
\begin{array}{rcl}
\p_t f & \!\!\!\!=\!\!\!\! & \mathrm{div}[ f \nabla(  (1+R) f +R g) ],\\[1ex]
\p_t g & \!\!\!\!=\!\!\!\!& \mu R\, \mathrm{div}[ g \nabla(  f + g ) ],
\end{array}
\right.
\qquad (t,x)\in (0,\infty)\times \R^d,\, d\geq1,
\end{equation}
which describes the motion of two thin fluid layers on an impermeable horizontal bottom, identified with the plane $\{x_{d+1}=0\}$, under the influence of gravity.
Here~${f(t,x)\geq0}$ is the thickness of the layer which has density  $\rho_-$ and viscosity $\mu_-$ and~${g(t,x)\geq0}$ is the thickness of the layer with density~$\rho_+$ and viscosity~$\mu_+$.
In particular, \eqref{P}  is a model for the spreading of two fluid blobs of different fluids on an impermeable surface. 
We assume that $\rho_->\rho_+>0$, the positive constants in~\eqref{P} being given by the relations
\[
R:=\frac{\rho_+}{\rho_--\rho_+} \qquad\text{and}\qquad \mu:=\frac{\mu_-}{\mu_+}.
\]
The system~\eqref{P} is derived in \cite{EMM12, WM2000, JM14} in a one-dimensional
setting and most of the analysis devoted so far to this problem, which we review now, is restricted to the one-dimensional case~${d=1}$. The well-posedness of \eqref{P} in the setting of classical solutions 
is addressed in \cite{EMM12}, while non-negative global weak solutions are constructed,
 by different approaches, in \cite{LM13, LM16x, ELM11, AJM20, BG19}. The rich dynamics described by the thin
  film Muskat problem~\eqref{P} is very well illustrated by the laboratory experiments described in \cite{WM2000}, but also by the numerical simulations reported in \cite{LM14x, ACCL19, Ou17}. 
  Besides, from a mathematical point of view, many of these experimental and numerical observations are rigorously established. In particular, non-negative weak solutions  
  to \eqref{P} possess finite speed of propagation and solutions emanating from certain  initial data feature the waiting time phenomenon, see \cite{LM16x}. 
  However, the finite speed of propagation property for a single fluid phase, that is, when only~$f(0)$ (or~$g(0)$) is compactly supported, is still an open problem. 
  When the system~\eqref{P} is posed on $\R$, the large time behavior of weak solutions is investigated in \cite{LM14x}. More precisely, it is shown that solutions starting from even initial data converge towards the (unique) even self-similar solution, a property which is in perfect agreement with the laboratory experiments reported in \cite{WM2000}. It is worthwhile to mention, as a special feature of the one-dimensional thin film Muskat problem~\eqref{P}, that,  depending  on the values of the parameters~$R$ and~$\mu$, there may exist a continuum of self-similar solutions which are not symmetric, see \cite{LM14x}, 
  and every weak solution to~\eqref{P} converges to one of these special solutions for large times. 
  Determining which self-similar solutions are attained in the large time limit is still an open problem,
   but numerical simulations performed in \cite{LM14x} seem to indicate that some of these non-symmetric self-similar solutions attract certain non-negative weak solutions. 
   The dynamics is much simpler in dimension $d=2$, as there exists only one  self-similar solution to~\eqref{P} which is radially symmetric and a global attractor, see \cite{ACCL19}. 
 
Though the thin film Muskat problem is formally derived as the singular limit of the Muskat problem when letting the thickness of the fluid layers vanish \cite{EMM12}, it is rather immediate to notice that \eqref{P} can be viewed as a two-phase generalization of the well-known porous medium equation 
 \begin{equation}\label{PME}
\p_t f =  \mathrm{div}(f \nabla f), \qquad (t,x)\in (0,\infty)\times \R^d,
\end{equation}
see \cite{Bou03, Va07}. Indeed, when $g=0$,  $f$ solves, up to a multiplicative factor which can be scaled out, the equation~\eqref{PME}. The goal of this paper is to establish the latter observation rigorously by performing the singular limit   
$$ 
\rho_+ \to 0,
$$
with $\mu$  kept constant or even letting $\mu\to \infty$, to recover the porous medium equation~\eqref{PME} in the limit, see Theorem~\ref{MT1}-Theorem~\ref{MT3} below. 
In order to present our results, we first quantify how~$\rho_+$ (and possibly also $\mu_+$)  vanishes by setting
\begin{align}\label{PC}
 R=\e\qquad\text{and}\qquad \mu=\mu(\e)=\frac{\ov\mu}{\e^\alpha},
\end{align}
with $\e\in(0,1) $,  $\alpha\in[0,\infty)$,  and  a  positive constant $\ov \mu$ (we  will   let $\e\to 0$).
For this choice of the parameters~$R$ and~$\mu$, the system~\eqref{P} becomes
\begin{equation}\label{Pe}
\left\{
\begin{array}{rcl}
\p_t f_\e & \!\!\!\!=\!\!\!\! & \mathrm{div}[ f_\e \nabla(  (1+\e) f_\e +\e g_\e) ],\\[1ex]
\p_t g_\e & \!\!\!\!=\!\!\!\!& \mu \e\, \mathrm{div}[ g_\e \nabla(  f_\e + g_\e ) ],
\end{array}
\right.
\qquad (t,x)\in (0,\infty)\times \R^d.
\end{equation}
We supplement \eqref{Pe}  with initial data
\begin{equation}\label{BC}
	\left( f_\e(0),g_\e(0) \right) = (f_0,g_0)\in \cK^2, 
\end{equation}
where
\begin{equation*}
\cK:=\Big\{h\in L_1(\R^d, (1+|x|^2) dx) \cap L_2(\R^d)\,:\, h\ge 0 \text{ a.e. and } \int_{\R^d} h(x)\, dx=1\Big\}.
\end{equation*}
We next recall that the system~\eqref{P} is a  gradient flow  for the energy functional
	\begin{equation}\label{EE}
		\cE_\e(f,g):=\frac{1}{2}\int_{\R^d}  [ f^2+\e(f+g)^2  ]\, dx
	\end{equation}
with respect to the $2$-Wasserstein distance, see \cite{LM13}, a similar property being available for the porous medium equation~\eqref{PME} as observed earlier in \cite{Ot01, Ot98}. Additionally, the entropy functional 
\begin{equation}\label{EF}
\mathcal{H}_{\e}(f,g) := \int_{\R^d}  \Big[ f\ln f + \frac{1}{\mu(\e)}g \ln g \Big]\, dx
\end{equation}
is also non-increasing along solutions to \eqref{P}. The gradient flow structure and the time monotonicity of \eqref{EF} are used in \cite{LM13} to construct non-negative global weak solutions to \eqref{P}. The next theorem just recalls the statement of \cite[Theorem~1.1]{LM13} and provides also the corresponding results in the case~$d\geq2$, as the strategy used in \cite{LM13} can be easily adapted to establish Theorem~\ref{T:1} in arbitrary space dimension~$d\geq1$.

\begin{thm}\label{T:1} Let  $\e\in(0,1)$, $(f_0,g_0)\in\cK^2,$ and assume \eqref{PC}.  Then, there exists  a pair $$(f_\e,g_\e):[0,\infty)\to  \cK^2$$ such that
\begin{itemize}
\item[(i)] $(f_\e,g_\e)\in L_\infty(0,\infty; L_2(\R^d;\R^2))$, $ (f_\e,g_\e) \in L_2(0,t;H^1(\R^d;\R^2))$,
\item[(ii)] $(f_\e,g_\e)\in C ([0,\infty),(W_{4}^{1}(\R^d;\R^2))')$ with $(f_\e,g_\e)(0)=(f_0,g_0),$
\end{itemize}
and $(f_\e,g_\e)$  is a weak solution to \eqref{Pe}-\eqref{BC} in the sense that 
\begin{subequations}\label{T1}   
\begin{align}
&\int_{\R^d} f_\e(t) \xi\, dx-\int_{\R^d} f_0 \xi\, dx + \int_0^t\int_{\R^d} f_\e  \nabla( (1+\e)f_\e+\e g_\e ) \cdot\nabla \xi\, dx\, ds=0,\label{T1a}\\[1ex]
&\int_{\R^d} g_\e(t) \xi\, dx - \int_{\R^d} g_0 \xi\, dx+ \mu\e\ \int_0^t \int_{\R^d} g_\e  \nabla(f_\e+ g_\e  ) \cdot\nabla\xi\, dx\, ds=0,\label{T1b}
\end{align}
\end{subequations}
for all $\xi\in C_c^\infty(\R^d) $ and $t\ge 0.$
 In addition, $(f_\e,g_\e)$ satisfy the following estimates:
\begin{align*}
{\rm (a)} \quad & \mathcal{H}_\e(f_\e(t),g_\e(t)) +\int_0^t\int_{\R^d} \big[  |\nabla f_\e|^2 + \e |\nabla(f_\e+g_\e)|^2  \big]\, dx\, ds\leq  \mathcal{H}_\e(f_0,g_0),\\[1ex]
{\rm (b)} \quad &\cE_\e(f_\e(t), g_\e(t))+ \frac{1}{2}\int_{0}^t\int_{\R^d} \big[f_\e |\nabla((1+\e)f_\e+\e g_\e)|^2+\mu\e^{2} g_\e|\nabla(f_\e+g_\e)|^2 \big]\, dx\, ds\leq \cE_\e(f_{0},g_{0})
\end{align*}
 for almost all $t\in(0,\infty)$.
\end{thm}

From now on $f_0$ and $g_0$ are fixed in $\cK$ and $(f_\e,g_\e)$, $\e\in(0,1)$, denotes the solution to the evolution problem~\eqref{Pe}-\eqref{BC} provided by Theorem~\ref{T:1}.
 In our first main result, see Theorem~\ref{MT1} below, 
we establish the convergence of the family~$(f_\e)_{\e\in (0,1)}$ found in Theorem~\ref{T:1} towards a weak solution to the porous medium equation~\eqref{PME} as~${\e\to 0}$ along a suitable sequence.

\begin{thm}\label{MT1} Let $d\geq 1$, $\alpha\in[0,\infty)$, and assume \eqref{PC}. There exists a sequence $(\e_k)_{k\ge 1}\subset(0,1)$   with~${\e_k\to0}$ and a   function~${f:[0,\infty)\to\cK}$ with
\begin{equation}
	\begin{split}
		&f\in L_\infty(0,\infty; L_2(\R^d))\cap C([0,\infty), (W_{4}^{1}(\R^d))'), \\
		&f \in  L_2(0,t; H^1(\R^d))\cap L_\infty(0,t;L_1(\R^d,|x|^2 dx)),\\
		&\sqrt{ f}\nabla f\in L_2((0,t)\times\R^d),
	\end{split} \label{X0}
\end{equation}
such that  $f_{\e_k}\to f$  in $L_2((0,t)\times \R^d)$ for all $t>0$ as $k\to\infty.$
Moreover, $f$ is a weak solution to the porous medium equation~\eqref{PME} determined by the initial condition $f_0$ in the sense that it satisfies~\eqref{X0} and 
\begin{equation*}\label{PMES}
 \int_{\R^d} f(t)  \xi\, dx-\int_{\R^d} f_0 \xi\, dx+\int_0^t\int_{\R^d} f \nabla f  \cdot\nabla\xi\, dx\, ds =0
\end{equation*}
for all $\xi\in C^\infty_c(\R^d)$ and $t\ge 0$.
\end{thm}

Even though there are several uniqueness results available for the porous medium equation~\eqref{PME} in the literature, see \cite{AL1983, BCP1984, BC1979, 0t1996, Pi1982, Va07} and the references therein, uniqueness of a weak solution to~\eqref{PME} in the sense of Theorem~\ref{MT1} does not seem to be dealt with and is thus reported below when the space dimension satisfies $d\leq 4$, see Theorem \ref{MT2}.
In fact, Theorem~\ref{MT2} improves Theorem~\ref{MT1} in space dimension $d\leq 4$ by providing rates for the convergence of the whole family $(f_\e)_{\varepsilon\in (0,1)}$ as~${\e\to0}$  towards the solution $f$ to the porous medium equation.

 \begin{thm}\label{MT2} Let $d\leq 4$, $\alpha\in[0,\infty)$, and assume \eqref{PC}.
 Then, the porous medium equation \eqref{PME} with initial data $f(0)=f_0$ has a unique solution $f$  in the sense of Theorem~\ref{MT1} and there exists a positive
  constant~$C=C(f_0,g_0,\ov \mu)$ such that
 \begin{align}\label{CR2}
 \|f_{\e}(t)-f(t)\|_{H^{-1}}\leq C e^{Ct} \e^{\frac{6d+36}{11d+36}}\qquad \text{for all  $t\geq0$ and $\e\in(0,1)$}.
\end{align}
\end{thm}

The dimension-dependent exponent of $\e$ featured in \eqref{CR2} is connected to the low regularity assumed on the solutions to~\eqref{PME} and~\eqref{Pe}. Under the additional assumption that an $\e$-independent $L_\infty$-bound is available for the solutions to~\eqref{PME} and~\eqref{Pe}, the outcome of Theorem~\ref{MT2} can be improved as follows.

\begin{cor}\label{C:2}
 Let $d\ge 1$, $\alpha\in [0,\infty)$, and assume \eqref{PC}. If $(f_0,g_0)\in L_\infty(\R^d;\R^2)$ and if there exists a positive constant~$\kappa$ such that
 \begin{equation}
 	\|f_\e(t)\|_\infty + \|f(t)\|_\infty \le \kappa, \qquad (t,\e)\in [0,\infty)\times (0,1), \label{X5}
 \end{equation}
then there exists a positive constant $C=C(f_0,g_0,\overline\mu,\kappa)$ such that
\begin{equation*}
	\|f_\e(t) - f(t)\|_{H^{-1}} \le C e^{Ct} \e \qquad\text{ for all $t\geq0$ and $\e\in (0,1)$.}
\end{equation*}
\end{cor}

The boundedness~\eqref{X5} is well-known for the solution to the porous medium equation~\eqref{PME}, 
as the comparison principle ensures that $\|f(t)\|_\infty \le \|f_0\|_\infty$ for $t\ge 0$. Such a bound is far from being obvious for solutions to~\eqref{Pe} and we refer to the forthcoming paper \cite{LMxx} for results in that direction.

Finally, in Theorem~\ref{MT3} we establish the convergence of the family $(g_\e)_{\e\in (0,1)}$ towards the initial condition~$g_0$ in the regime where $\alpha\in[0,1/(d+2))$. 

 \begin{thm}\label{MT3} Let $d\geq 1$,  $\alpha\in[0,1/(d+2))$,  and assume \eqref{PC}.
Then, there exists a positive constant~$C=C(f_0,g_0,\ov\mu)$ such that
 \begin{align*} 
 \|g_{\e}(t)-g_0\|_{ H^{-1-d}}\leq C(1+t) \e^{\tfrac{1}{d+2}-\alpha}\qquad \text{for all $t\geq0$ and $\e\in(0,1)$}.
\end{align*}
\end{thm}

The outline of the paper is as follows. In Section~\ref{Sec:2} we deduce from Theorem~\ref{T:1} a handful of estimates for the solutions $(f_\e,g_\e)$ to \eqref{Pe}-\eqref{BC} which form the basis of the proof of the 
convergence result stated in Theorem~\ref{MT1}. Section~\ref{Sec:3} is next devoted to the proofs of Theorem~\ref{MT2} and Corollary~\ref{C:2}, which use the estimates from Section~\ref{Sec:2}, a duality technique, and  Gronwall's lemma. Finally, in Section~\ref{Sec:4}, we establish Theorem~\ref{MT3}, using once more the estimates established in Section~\ref{Sec:2}.

\section{$\e$-independent estimates and proof of Theorem~\ref{MT1}} \label{Sec:2}
To begin with, we  derive from Theorem~\ref{T:1} estimates for the solutions $(f_\e,g_\e)$ to~\eqref{Pe}-\eqref{BC}, see Lemma~\ref{L:1}.
These estimates, together with Lemma~\ref{L:AL1} and a classical compactness result \cite[Corollary~4]{Si87}, enable us to establish the convergence
 of $(f_\e)_{\e\in (0,1)}$  along a sequence $\e_k\to 0$ towards the  solution to the porous medium equation~\eqref{PME}, see Lemma~\ref{L:2}. 
 We conclude the section with the proof of Theorem~\ref{MT1}. In the following we use the shorthand notation
$$
h_\e:=f_\e+g_\e,\qquad \e\in(0,1). 
$$
\pagebreak

\begin{lemma}\label{L:1} Let $d\geq 1$, $\alpha\in[0,\infty)$,  and assume \eqref{PC}. Then:
\begin{itemize}
 \item[(i)]   $(f_\e)_{\e\in(0,1)}$ and  $(\sqrt{\e} h_\e)_{\e\in(0,1)}$ are bounded in $L_\infty(0,\infty; L_1(\mathbb{R}^d)\cap L_2(\R^d));$ \\[-2ex]
   \item[(ii)]   $\big( \sqrt{f_\e}\nabla(f_\e+\e h_\e) \big)_{\e\in(0,1)}$ is bounded in $L_2((0,\infty)\times\R^d);$ 
   \item[(iii)] $(\nabla f_\e)_{\e\in(0,1)}$ and  $(\sqrt{\e}\nabla h_\e)_{\e\in(0,1)}$ are bounded in $L_2((0,t)\times\R^d) $ for all $t\geq0$;  \\[-2ex]
    \item[(iv)] $(\p_t f_\e)_{\e\in(0,1)}$ is bounded in $L_2(0,\infty;(W^1_4(\R^d))') $;
    \item[(v)] $(f_\e)_{\e\in(0,1)}$ is bounded in $L_\infty(0,t; L_1(\mathbb{R}^d,|x|^2\,dx))$ for all $t\geq 0$.
\end{itemize}
\end{lemma}

\begin{proof}
Since $\e\in (0,1)$ and
 $$
 \mathcal{E}_\e(f_0,g_0) \le \frac{\|f_0\|_2^2 + \|f_0+g_0\|_2^2}{2},
 $$
the estimates (i)-(ii) directly follow from the definition of $\mathcal{K}$ and the energy inequality, see Theorem~\ref{T:1}~(b).
 
In order to prove (iii), let us consider, for  each $n\in\N$, the  function $\xi_n:\R^d\to\R$ defined by
 $$
 \xi_n(x):=
 \left\{
 \begin{array}{clll}
 |x|^2&,& |x|\leq n,\\
 4n|x|-|x|^2-2n^2&,& n\leq |x|\leq 2n,\\
 2n^2&,& |x|\geq 2n.
 \end{array}
 \right. 
 $$
 We point out that $\xi_n-2n^2$ is continuously differentiable and has compact support, hence it can be approximated in the $W^1_\infty$-norm by functions in $C_c^\infty(\R)$.
Moreover, the properties of the solutions to \eqref{Pe}-\eqref{BC} listed in Theorem~\ref{T:1}  enable us to show that 
\begin{align*}
 & \lim_{n\to\infty} \int_{\R^d} f_\e(t,x)\xi_n(x)\, dx = \int_{\R^d} f_\e(t,x)|x|^2\, dx,\\[1ex] 
 & \lim_{n\to\infty} \int_{\R^d} g_\e(t,x)\xi_n(x)\, dx =  \int_{\R^d} g_\e(t,x)|x|^2\, dx,\\[1ex]
 & \lim_{n\to\infty} \int_0^t\int_{\R^d} f_\e  \nabla(f_\e+\e h_\e )\cdot  \nabla\xi_n\, dx\, ds = 2\int_0^t\int_{\R^d} f_\e  \nabla( f_\e+\e h_\e ) \cdot x\, dx,\\[1ex]
 & \lim_{n\to\infty} \int_0^t \int_{\R^d} g_\e  \nabla h_\e\cdot \nabla \xi_n\, dx\, ds = 2\int_0^t \int_{\R^d} g_\e  \nabla h_\e \cdot x\, dx\, ds,\qquad t\geq0.
\end{align*}
Hence, using $\xi_n-2n^2$, $n\in\N$, as test functions in \eqref{T1}, these convergences   yield in the limit~${n\to\infty}$ that 
 \begin{align*} 
  \int_{\R^d}\Big(f_\e(t)+\frac{1}{\mu} g_\e(t)\Big)|x|^2\, dx- d\int_0^t\int_{\R^d} \left( f_\e^2+\e h_\e^2 \right)\, dx\, ds=\int_{\R^d}\Big(f_0+\frac{1}{\mu} g_0\Big)|x|^2\, dx,\quad t\geq0.
 \end{align*}
In particular, it follows from \eqref{PC} and Theorem~\ref{T:1}~(b) that there exists a constant~${C=C(f_0,g_0,\ov\mu)>0}$ such that 
 \begin{align}\label{ME1}
  \int_{\R^d}\Big( f_\e(t) + \frac{1}{\mu} g_\e(t)\Big) |x|^2\, dx\leq C(1+t),\quad t\geq0.
 \end{align}
 Taking advantage of \cite[Lemma~A.1]{LM13}, we find a positive universal constant $C_E$ such that
 $$
 -\mathcal{H}_\e(f_\e(t),g_\e(t))\leq C_E \Big(1+\frac{1}{\mu}\Big)+\int_{\R^d}\Big(f_\e(t)+\frac{1}{\mu} g_\e(t)\Big)|x|^2\, dx,\qquad t\geq0,
 $$
 which, together with \eqref{ME1} and the entropy estimate  in Theorem~\ref{T:1}~(a), shows that 
 \begin{equation*}
 \int_0^t\int_{\R^d} \big [ |\nabla f_\e|^2 + \e |\nabla h_\e|^2  \big]\, dx\, ds\leq \mathcal{H}_\e(f_0,g_0) + C(1+t),\qquad t\geq0.
 \end{equation*}
Since $(f_0,g_0)\in\cK^2$ and $r\ln{r}\le r+r^2$ for $r\ge 0$, we conclude, together with  \eqref{PC}, that there exists a constant~${C=C(f_0,g_0,\ov\mu)>0}$ with the property that
 \begin{equation}\label{eq:DED}
 \int_0^t\int_{\R^d} \big [ |\nabla f_\e|^2 + \e |\nabla h_\e|^2  \big]\, dx\, ds\leq  C(1+t),\qquad t\geq0,
 \end{equation}
and (iii) follows from the above inequality.

Next, a classical consequence of Theorem~\ref{T:1} (see Lemma~\ref{L:43} below for a related result) ensures that $(f_\e,g_\e)$ solves~\eqref{Pe} in  distributional sense; that is,
 \begin{equation}\label{DS}
 \p_t f_\e=\mathrm{div} J_{f_\e}\quad\text{and}\quad \p_t g_\e=\mathrm{div}  J_{g_\e}\qquad\text{in $\mathcal{D}'((0,\infty)\times\R^d)$,}
 \end{equation}
 where the fluxes $J_{f_\e}$ and $J_{g_\e}$ are given by
 \[
 J_{f_\e}:=f_\e  \nabla( f_\e+\e h_\e )\quad\text{and}\quad J_{g_\e}:=\mu\e g_\e   \nabla h_\e.
 \]
The estimates (i)-(ii) from Theorem~\ref{T:1}, along with H\"older's inequality, lead us to
\begin{align*}
 \|J_{f_\e}\|_{L_2(0,\infty;L_{4/3}(\R^d))}^2&=\int_0^\infty\Big(\int_{\R^d} \big|   f_\e \nabla(f_\e+\e h_\e)\big|^{4/3}\, dx\Big) ^{3/2}\, dt\\[1ex]
 &= \int_0^\infty\Big(\int_{\R^d} \big|  f_\e^2\big|^{1/3} \big| \sqrt{f_\e}\nabla(f_\e+\e h_\e)\big|^{4/3}\, dx\Big) ^{3/2}\, dt\\[1ex]
 &\leq \int_0^\infty \|f_\e\|_2 \left\|  \sqrt{f_\e}\nabla( f_\e+\e h_\e) \right\|_2^2 \, dt\\[1ex]
 &\leq   \|f_\e\|_{L_\infty(0,\infty; L_2(\R^d))} \left\|\sqrt{f_\e}\nabla( f_\e+\e h_\e) \right\|_{L_2((0,\infty)\times \R^d)}^2,
\end{align*}
which shows that $( J_{f_\e})_{\e\in(0,1)} $ is bounded in $L_2(0,\infty; L_{4/3}(\R^d))$. 
 This property  immediately implies~(iv) by a duality argument.
 
Finally, the bound~(v) is a straightforward consequence of  \eqref{ME1}.
 \end{proof}

The next step is the continuity and compactness of some embeddings involving weighted~${\text{$L_p$-spaces}}$, 
which will serve when establishing the convergence of the family $(f_\e)_{\e\in (0,1)}$ (along a suitable sequence~$\e_k\to 0$).

\begin{lemma}\label{L:AL1}
	\phantom{a}
	
\begin{itemize}
\item[(i)] Given $p\in(2,\infty]$, the embedding $L_p(\R^d)\cap L_1(\R^d, |x|^2dx)\hookrightarrow L_2(\R^d, |x|^{\frac{2(p-2)}{p-1}}dx)$ is continuous. \\[-1ex]
\item[(ii)] The embedding $H^1(\R^d)\cap L_1(\R^d, |x|^2dx)\hookrightarrow L_1(\R^d)\cap L_2(\R^d)$ is compact.
\end{itemize}
\end{lemma}

\begin{proof}
The claim~(i) with $p=\infty$ is obvious.  For $p\in(2,\infty)$, H\"older's inequality leads us to
\begin{equation}\label{X1}
\int_{\R^d} |f(x)|^2|x|^{\frac{2(p-2)}{p-1}}dx\leq \Big(\int_{\R^d}|f(x)|\, |x|^{2}\,dx\Big)^{\frac{p-2}{p-1}}\|f\|_p^{\frac{p}{p-1}}
\end{equation}
and (i) follows.

With respect to (ii), let $(f_n)_{n\ge 1}$ be a bounded sequence in $H^1(\R^d)\cap L_1(\R^d, |x|^2dx)$ and set 
\[
	M := \sup_{n\ge 1}\big\{ \|f_n\|_{H^1} + \|f_n\|_{L_1(\R^d,|x|^{2}\,dx)}\big\}.
\]
Owing to the compactness of the embedding of $H^1(B_R(0))$ in $L_2(B_R(0))$ for any $R>0$, where we set~${B_R(0):=\{x\in\R^d\,:\,\ |x|<R\}}$, it follows from a standard Cantor diagonal procedure that there exist a function~${f\in H^1(\R^d)\cap L_1(\R^d, |x|^2dx)}$ and a subsequence  of $(f_n)_{n\ge 1}$ (not relabeled) such that~${f_n\to f}$  in $L_2(B_R(0))$ for all $R>0$ and
\[
	\|f\|_{H^1} + \|f\|_{L_1(\R^d,|x|^{2}\,dx)}\le M.
\]
Choosing $p\in (2,\infty)$ such that~${p<2d/(d-2)}$ when~$d\geq 3$, Sobolev's embedding and \eqref{X1} then lead us to
\begin{align*} 
\|f_n-f\|_2^2 &\leq \int_{B_R(0)}|f_n(x)-f(x)|^2\,dx+ R^{-\frac{2(p-2)}{p-1}}\int_{\{|x|>R\}}|f_n(x)-f(x)|^2|x|^{\frac{2(p-2)}{p-1}}\,dx\\[1ex]
&\leq \|f_n-f\|^2_{L_2(B_R(0))}+R^{-\frac{2(p-2)}{p-1}} \|f_n-f\|_{L_1(\R^d,|x|^{2}\,dx)}^{\frac{p-2}{p-1}}\|f_n-f\|_p^{\frac{p}{p-1}} \\
&\leq \|f_n-f\|^2_{L_2(B_R(0))}+CR^{-\frac{2(p-2)}{p-1}} (2M)^{\frac{p-2}{p-1}}\|f_n-f\|_{H^1}^{\frac{p}{p-1}} \\
&\leq \|f_n-f\|^2_{L_2(B_R(0))}+ C M^2R^{-\frac{2(p-2)}{p-1}} 
\end{align*}
for all $n\in\N$ and $R>1$. Letting first $n\to\infty$ and then $R\to\infty$, we deduce that $f_n\to f$ in $L_2(\R^d)$. Finally, for $R>1$,
\begin{align*}
\|f_n-f\|_1& \leq  \|f_n-f\|_{L_1(B_R(0))} + R^{-2}\|f_n-f\|_{L_1(\R^d,|x|^{2}\,dx)} \\
& \leq \sqrt{|B_R(0)|} \|f_n-f\|_{L_2(B_R(0))} + 2MR^{-2}.
\end{align*}
Arguing as before completes the proof of~(ii).
\end{proof}

We now use the  estimates derived in Lemma~\ref{L:1} and the previous result to  deduce the following convergences.

\begin{lemma}\label{L:2} 
Let $d\geq 1$, $\alpha\in[0,\infty)$,  and assume \eqref{PC}.
 There exists a sequence $(\e_k)_{k\ge 1}\subset(0,1)$ with~$\e_k\to 0$ and a function~${f\in L_2(0,t; H^1(\R^d))}\cap L_\infty(0,\infty;L_2(\R^d))$ for all $t>0$ such that
 \begin{itemize}
  \item[(i)] $f_{\e_k}\to f$ in $L_2((0,t)\times \R^d)$ and in $C([0,t], (W^1_4(\R^d))')$ for all $t>0$;
  \item[(ii)] $\nabla f_{\e_k}\rightharpoonup \nabla f$  in $L_2((0,t)\times \R^d)$ for all $t>0$;
   \item[(iii)] $\p_t f_{\e_k}\rightharpoonup \p_t f$ in $L_2(0,\infty;(W^1_4(\R^d))');$ 
   \item[(iv)] $\sqrt{f_{\e_k}}\nabla(f_{\e_k}+\e_k h_{\e_k})\rightharpoonup \sqrt{f}\nabla f$ in $L_{2}((0,t)\times\R^d) $  for all $t>0$.
 \end{itemize}
\end{lemma}

\begin{proof} Let $t>0$ be fixed.
 In view of Lemma~\ref{L:1}~(i),~(iii), and~(v), we obtain the boundedness of $(f_\e )_{\e\in(0,1)}$ in~${L_2(0,t;  H^1(\R^d))}$ and in $L_\infty(0,t; L_1(\R^d, |x|^2\, dx)) $, as well as that of $(\p_t f_\e )_{\e\in(0,1)}$ in $L_2(0,t;  (W^1_4(\R^d))'  )$. Hence, the family~$( f_\e )_{\e\in(0,1)}$ is  bounded in $L_2(0,t;   H^1(\R^d)\cap L_1(\R^d, |x|^2\,  dx))$. Since 
$$
H^1(\R^d)\cap L_1(\R^d, |x|^2\,  dx)\hookrightarrow L_1(\R^d)\cap L_2(\R^d)\hookrightarrow(W^1_4(\R^d))'
$$
and the first embedding is compact according to Lemma~\ref{L:AL1}~(ii), we infer from a classical compactness result, see \cite[Corollary~4]{Si87}, and a Cantor diagonal argument that there exist a sequence $\e_k\to0$ and a function $f$ such that ${f_{\e_k}\to f}$ in $L_2((0,t)\times \R^d)$ and in $C([0,t], (W^1_4(\R^d))')$ for all $t>0$. Note that  Lemma~\ref{L:1}~(i) now immediately implies that~$f\in  L_\infty(0,\infty;L_1(\R^d)\cap L_2(\R^d))$. 
 
Next, the convergence~(ii)  is a straightforward consequence of  Lemma~\ref{L:1}~(iii) and the just established Lemma~\ref{L:2}~(i), from which we also deduce that ${f\in L_2(0,t; H^1(\R^d))} $ for all $t>0$.
  
 Recalling Lemma~\ref{L:1}~(iv), the convergence $\p_tf_{\e_k}\rightharpoonup \p_tf$  in $L_2(0,t;(W^1_4(\R))')$  
 (after possibly extracting a further subsequence) follows  from Lemma~\ref{L:2}~(i) and the reflexivity of~${L_2(0,t;(W^1_4(\R^d))')}$. 
 
 With respect to (iv), due to  Lemma~\ref{L:1} (ii), we may assume that  there is $\mathbf{F}\in L_2((0,\infty)\times\R^d)$ such that
 \begin{equation}\label{X2}
 	\text{$\sqrt{f_{\e_k}}\nabla (f_{\e_k}+\e_kh_{\e_k}) \rightharpoonup \mathbf{F}$ \qquad in $L_2((0,\infty)\times\R^d)$.} 
 \end{equation}
Moreover,  the bounds in Lemma \ref{L:1}~(iii), along with Lemma~\ref{L:2}~(ii),    imply that   
  \[
 \text{$\nabla (f_{\e_k}+\e_kh_{\e_k})\rightharpoonup \nabla f$ \qquad in $L_2((0,t)\times\R^d),$}
 \]
 while Lemma~\ref{L:2}~(i) leads us to
 \begin{align}\label{CL4}
 \text{$\sqrt{f_{\e_k}}\to \sqrt{f}$ \qquad in $ L_4((0,t)\times\R^d).$}
 \end{align}
 Combining the last two convergences we find
 \[
 \text{$\sqrt{f_{\e_k}}\nabla(f_{\e_k}+\e_k h_{\e_k})\rightharpoonup \sqrt{f}\nabla f$ in $L_{4/3}((0,t)\times\R^d)$.}
 \]
Recalling \eqref{X2}, we conclude that $\mathbf{F} = \sqrt{f}\nabla f$ and therewith establish~(iv).
 \end{proof}

We are now in a position to prove Theorem~\ref{MT1}.
 
 \begin{proof}[Proof of Theorem~\ref{MT1}]
 Let $(\e_k)_k\subset(0,1)$ and $f$  be as found in Lemma~\ref{L:2}.
As a direct consequence of Lemma~\ref{L:1}~(v) and Lemma~\ref{L:2}~(i), which in particular implies that  $f_{\e_k}(t)\to f(t)$ in $L_2(\R^d)$ for almost all $t>0$,
we deduce that $f\in   L_\infty(0,t;L_1(\R^d,(1+|x|^2) dx))$ for all $t>0$.
Taking also into account that~${f\in C([0,\infty), (W^1_4(\R^d))'),}$ it follows that $f(t)\geq 0$  a.e. in $\R^d$ for all $t\geq0.$ Moreover, in view of Lemma~\ref{L:2}~(i) and~(iv), it is straightforward to pass to the limit $k\to \infty$ in \eqref{T1a} with~${\e=\e_k}$ and obtain
\begin{equation*}
	\int_{\R^d} f(t)  \xi\, dx-\int_{\R^d} f_0 \xi\, dx+\int_0^t\int_{\R^d} f \nabla f \cdot  \nabla\xi\, dx\, ds=0 
\end{equation*}
for all $\xi\in C^\infty_c(\R^d)$ and all $t\ge 0$, thereby establishing \eqref{PMES}.

Finally, choosing~${\zeta_n:\R^d \to\R}$ with
\[
\zeta_n(x):= 
 \left\{
 \begin{array}{clll}
1&,& |x|\leq n,\\
  n+1-|x|&,& n\leq |x|\leq n+1,\\
   0&,&   |x|\geq n+1,
 \end{array}
 \right. \qquad n\in\N,
\]
as test function in \eqref{PMES}, we find that~${\|f(t)\|_1= \|f_0\|_1=1}$ for all~$t\geq0$. 
Therefore $f(t)\in\cK$ for all~$t\ge 0$.
 \end{proof}

\section{Estimating   the error $\|f_\e(t)-f(t)\|_{H^{-1}}$}\label{Sec:3}

In this section we restrict our arguments to  the case when the space dimension satisfies~$d\leq 4$. The main goal is to provide an estimate  for the error  $\|f_\e(t)-f(t)\|_{H^{-1}}$ with $t\geq0$, cf. Theorem~\ref{MT2}.
In particular, we also prove that the porous medium equation  \eqref{PME} with initial data $f_0\in\cK$ has a unique solution in the sense of Theorem~\ref{MT1}.
Hence, this improves Theorem~\ref{MT1} in the sense that now the whole family $(f_\e)_{\e\in(0,1)}$ converges for $\e\to0$ towards the corresponding solution to the porous medium equation. In order to prepare the proof of Theorem~\ref{MT2}, which we postpone to the end of the section, we first introduce some notation. We recall that  $1-\Delta:W^2_p(\R^d)\to L_p(\R^d)$ is an isomorphism for all $p\in (1,\infty)$ and 
$$
\|f\|_{H^{-1}}= \|(1-\Delta)^{-1}[f]\|_{H^1} \qquad \text{for all $f\in L_2(\R^d)$}.
$$
Theorem~\ref{MT2} then amounts to estimate the norm $\|F_\e(t)-F(t)\|_{H^{1}}$, where 
\[
F_\e:=(1-\Delta)^{-1}[f_\e]\qquad \text{and}\qquad F:=(1-\Delta)^{-1}[f].
\] 
Testing the equations \eqref{T1a} and \eqref{PMES} by $(1-\Delta)^{-1}[\xi]$, with $\xi\in C^\infty_c(\R^d)$ and using the self-adjointness of $(1-\Delta)^{-1}$, we arrive  at
\begin{align*}
\int_{\R^d} F_\e(t)  \xi\, dx & - \int_{\R^d} F_0 \xi\, dx \\
& + \frac{1}{2}\int_0^t\int_{\R^d} \Big[ f_\e^2\xi -\xi(1-\Delta)^{-1}[f_\e^2] +2\e(1-\Delta)^{-1}[f_\e\nabla h_\e]\cdot \nabla\xi \Big] \, dx\, ds =0,\\[1ex]
\int_{\R^d} F(t)  \xi\, dx & -\int_{\R^d} F_0 \xi\, dx+  \frac{1}{2} \int_0^t\int_{\R^d}      \Big[ f^2\xi -\xi(1-\Delta)^{-1}[f^2] \Big] \, dx\, ds =0,
\end{align*}
where  $F_0:=(1-\Delta)^{-1}[f_0].$
Hence, letting 
$$d_\e=f_\e-f\qquad\text{and}\qquad D_\e:=F_\e-F=(1-\Delta)^{-1}[d_\e],$$
 we have $d_\e(0)=D_\e(0)=0$ and, subtracting the above identities, we  deduce that
\begin{equation}\label{II1}
\begin{split}
\int_{\R^d} D_\e(t) \xi\, dx & + \frac{1}{2} \int_0^t\int_{\R^d} \Big[ d_\e (f_\e+f) \xi  -\xi(1-\Delta)^{-1}[d_\e(f_\e+f)] \Big]\, dx\, ds  \\[1ex]
& + \e \int_0^t\int_{\R^d} (1-\Delta)^{-1}[f_\e\nabla h_\e]\cdot \nabla\xi \, dx\, ds =0.
\end{split}
\end{equation}

As a consequence of~\eqref{II1}, Theorem~\ref{T:1}, and Theorem~\ref{MT1}, we obtain  the following result.

\begin{lemma}\label{L:43}
 Let $t>0$ and $\xi\in C_c^\infty([0,t]\times\R^d)$. Then
 \begin{equation}\label{II2}
 \begin{aligned}
  &\hspace{-0.25cm}\int_{\R^d} D_\e(t)\xi(t)\, dx-\int_0^t\int_{\R^d} D_\e \p_t\xi\, dx\, ds\\[1ex]
  &=- \frac{1}{2} \int_0^t\int_{\R^d} \Big[ d_\e (f_\e+f) \xi  -\xi(1-\Delta)^{-1}[d_\e (f_\e+f)] + 2\e(1-\Delta)^{-1}[f_\e\nabla h_\e]\cdot \nabla\xi\Big]\, dx\, ds
 \end{aligned}
 \end{equation}
 and
  \begin{equation}\label{II3}
  \int_{\R^d} d_\e(t)\xi(t)\, dx -\int_0^t\int_{\R^d} d_\e \p_t\xi\, dx\, ds= \int_0^t\int_{\R^d} \Big[ f \nabla f - f_\varepsilon \nabla (f_\e +\e h_\e) \Big]\cdot \nabla\xi\, dx\, ds
 \end{equation}
 for all $\e\in(0,1)$.
\end{lemma}

\begin{proof} 
This is a classical result and therefore we omit its proof.
\end{proof}

The next lemma provides an integral identity, cf.~\eqref{II4}, which is obtained when formally testing with $D_\e$ in~\eqref{II3}. This identity is the starting point in the proof of Theorem~\ref{MT2}.

\begin{lemma}\label{L:44}
For all $t\geq 0$ and $\e\in(0,1)$,
  \begin{equation}\label{II4}
\|d_\e(t)\|_{H^{-1}}^2=\|D_\e(t)\|_{H^1}^2=\int_0^t\int_{\R^d}\Big[-d_\e^2(f_\e+f) +D_\e d_\e(f_\e+f) +2\e f_\e\nabla h_\e \cdot \nabla D_\e\Big]\, dx\, ds.
 \end{equation}
 \end{lemma}

\begin{proof}
Since  $D_\e=(1-\Delta)^{-1}[d_\e]$ with $d_\e\in L_\infty(0,\infty; L_2(\R^{d}))$, we have~${D_\e\in  L_\infty(0,\infty; H^2(\R^{d}))}$. Let $\tilde D_\e\in L_\infty(\R;H^2(\R^d))$ denote the even reflection of $D_\e$ with respect to the boundary $\{t=0\}$ of~${(0,\infty)\times\R^d}$ and choose~${\varphi\in C^\infty_c(\R^d,[0,1])}.$
We then define~$\xi_\delta\in C^\infty_c(\R^{d+1})$ as the convolution~${\xi_\delta:=\rho_\delta*(\varphi \tilde D_\e),}$
 where~${(\rho_\delta)_{\delta\in(0,1)}}$ is a standard mollifier on $\R^{d+1}$. 
 We first note that there exists a positive integer $N\in\N$ with the property that~${\supp \xi_\delta(t) \subset[-N,N]^d}$ for all $\delta\in(0,1)$ and~${t\geq 0}$.  
Using~$\xi_\delta$ as a test function in~\eqref{II3}, we get
  \begin{equation}\label{II3'}
 \begin{aligned}
  \int_{\R^d} d_\e(t)\xi_\delta(t)\, dx & -\int_0^t\int_{\R^d} d_\e \p_t\xi_\delta\, dx\, ds\\[1ex]
  &=\int_0^t\int_{\R^d} \Big[f\nabla f -f_\e\nabla(f_\e+\e h_\e)\Big]\cdot\nabla\xi_\delta\, dx\, ds\qquad\text{for all $t>0$.}
 \end{aligned}
 \end{equation}
 To pass to the limit~$\delta\to0$ in the term on the right-hand side of \eqref{II3'} we recall that 
 $$
 f\nabla f -f_\e\nabla(f_\e+\e h_\e)\in L_{4/3}((0,t)\times\R^d)
 $$
 by Lemma~\ref{L:1}~(i)-(ii), Theorem~\ref{MT2}, and H\"older's inequality.
 Moreover, the regularity of $\tilde{D}_\e$, the restriction to space dimension $d\le 4$, and Sobolev's embedding ensure that $\varphi \tilde D_\e \in L_\infty(\R;W^1_4(\R^{d})),$ from which we deduce that $\nabla \xi_\delta\to \nabla (\varphi D_\e)$ in $L_{4}((0,t)\times\R^d) $ for all $t>0$ when letting~$\delta\to0$. 
 We therewith get
\begin{equation*} 
\begin{split}
& \lim_{\delta\to 0} \int_0^t\int_{\R^d} \Big[f\nabla f -f_\e\nabla(f_\e+\e h_\e)\Big]\cdot\nabla\xi_\delta \, dx\, ds \\
& \hspace{2cm} =  \int_0^t\int_{\R^d} \Big[f\nabla f  - f_\e\nabla(f_\e+\e h_\e)\Big]\cdot\nabla (\varphi D_\e) \, dx\, ds.
\end{split}
\end{equation*}
 
To deal with the first term on the left-hand side of the equality \eqref{II3'}, we note that  
 $d_\e \xi_\delta \to   \varphi  d_\e D_\e$ in~${L_1((0,t)\times\R^d)}$ for $\delta\to0$ and all $t>0$,
 and therefore  we may assume that 
 \begin{equation*} 
 \lim_{\delta\to 0} \int_{\R^d} d_\e(t)\xi_\delta(t)\, dx =  \int_{\R^d} \varphi  d_\e(t) D_\e(t)\, dx \qquad\text{for almost all $t>0$.}
 \end{equation*}
 
 To pass to the limit $\delta\to0$ in the second term on the left-hand side of \eqref{II3'}, we note that   Lemma~\ref{L:2}~(iii) 
 guarantees that $\partial_t d_\e\in L_2(0,\infty,(W_{4}^1(\R^d))')$. 
 Since $\partial_t D_\e = (1-\Delta)^{-1}[\partial_t d_\e]$, we have~$\partial_t D_\e\in L_2(0,\infty,W_{4/3}^1(\R^d))$ and this implies that $\partial_t \tilde{D}_\e$ lies in $L_2(\R;L_{4/3}(\R^d))$.
It follows that~${\p_t\xi_\delta=\rho_\delta*(\varphi \p_t\tilde D_\e)}$ in $\R^{d+1}$ and therefore~$\p_t\xi_\delta\to \varphi \p_t D_\e$ in $L_2(0,t; L_{4/3}(\R^d))$ as $\delta\to0$.
Taking also into account that~${d_\e\in  L_2(0,t; L_{4}(\R^d))}$, due to Lemma~\ref{L:2} (i)-(ii) and the continuous embedding of $H^1(\R^d)$ in $L_4(\R^d)$,
we may now pass to the limit $\delta \to 0$  in~\eqref{II3'} to conclude that the identity
  \begin{equation*} 
	\begin{aligned}
	\int_{\R^d}\varphi d_\e(t)  D_\e(t)\, dx & -\int_0^t\int_{\R^d}\varphi d_\e \p_t D_\e\, dx\, ds\\[1ex]
	&=\int_0^t\int_{\R^d} \Big[ f \nabla f  - f_\e \nabla( f_\e+ \e h_\e)\Big]\cdot \nabla (\varphi D_\e)\, dx\, ds
	\end{aligned}
\end{equation*}
is satisfied for almost all $t\in (0,\infty)$. Now, owing to \eqref{II2},
$$
\p_t D_\e = - \frac{d_\e(f_\e+f)}{2} + (1-\Delta)^{-1}\left[ \frac{d_\e(f_\e+f)}{2} \right] + \e \mathrm{div}\left( (1-\Delta)^{-1}[f_\e \nabla h_\e] \right)
$$
in $\mathcal{D}'((0,\infty)\times\R^d)$. Combining the above two identities, we end up with
  \begin{equation*} 
 \begin{aligned}
  &\int_{\R^d}\varphi d_\e(t)  D_\e(t)\, dx\\[1ex]
  &=\int_0^t\int_{\R^d}\varphi d_\e\left( -\frac{d_\e (f_\e+f)}{2}
  +(1-\Delta)^{-1}\left[ \frac{d_\e (f_\e+f)}{2} \right] + \e\nabla\cdot (1-\Delta)^{-1}[f_\e\nabla h_\e] \right)\, dx\, ds\\[1ex]
  &\hspace{0.5cm}-\int_0^t\int_{\R^d} \left[ \nabla\left( \frac{d_\e (f_\e+f)}{2} \right) +\e f_\e\nabla h_\e \right]\cdot \nabla (\varphi D_\e)\, dx\, ds
 \end{aligned}
 \end{equation*}
for almost every $t\in (0,\infty)$.
Choosing   a suitable
approximating sequence $(\varphi_n)_n \subset  C_c^\infty(\R^d,[0,1])$  for the constant function $1$, we infer from the above identity, after passing to the limit $n\to\infty$
 and using the relation $D_\e =(1-\Delta)^{-1}[d_\e]$, that~\eqref{II4} holds true (for all $t\geq 0$ due to the fact that~${D_\e\in C([0,\infty); W_{4/3}^1(\R^d))}$). 
\end{proof}

We are now in a position to establish our second main result, see Theorem~\ref{MT2}.

\begin{proof}[Proof of Theorem~\ref{MT2}]
Given $t\geq 0$, it follows from  \eqref{II4}  that 
  \begin{equation}\label{II5}
 \|D_\e(t)\|_{H^1}^2+\int_0^t\int_{\R^d}d_\e^2(f_\e+f) \, dx\, ds\leq T_1+T_2,
 \end{equation}
 where
 \begin{equation*}
  T_1:= \int_0^t\int_{\R^d}D_\e d_\e (f_\e+f)\, dx\, ds\qquad\text{and }\qquad  T_2:= 2\e \int_0^t\int_{\R^d} f_\e\nabla h_\e \cdot \nabla D_\e \, dx\, ds.
 \end{equation*}
Below we estimate the terms $T_1$ and $T_2$ separately  in order  to obtain  an integral inequality to which we may apply Gronwall's inequality and conclude in this way our  claim~\eqref{CR2}.\medskip

\noindent{\bf The term $T_1$.} Using H\"older's and Young's inequalities, we find
 \begin{align}
  T_1&\leq\left( \int_0^t\int_{\R^d} d_\e^2 (f_\e+f)\, dx\, ds \right)^{1/2}\left( \int_0^t\int_{\R^d} D_\e^2 (f_\e+f)\, dx\, ds \right)^{1/2} \nonumber\\[1ex]
  &\leq \frac{1}{4}\int_0^t\int_{\R^d} d_\e^2(f_\e+f) \, dx\, ds+\int_0^t\int_{\R^d} D_\e^2 (f_\e+f)\, dx\, ds. \label{E1}
 \end{align}
The first term on the right-hand of~\eqref{E1} is clearly controlled by the left-hand side of \eqref{II5} and we are thus left with estimating the second term. In view of the continuous embedding $H^1(\R^d)\hookrightarrow L_{4}(\R^d)$ (recall that $d\le 4$), we infer from H\"older's  inequality that
\begin{equation*}
 \int_0^t\int_{\R^d} D_\e^2 (f_\e+f)\, dx\, ds \leq \int_0^t \|D_\e\|_{4}^2 \|f_\e+f\|_{2} \, ds\leq C \int_0^t \|f_\e+f\|_{2} \|D_\e\|_{H^1}^2\, ds.
\end{equation*}
Hence, in view of  Theorem~\ref{MT1} and Lemma~\ref{L:1}~(i), we have 
\begin{equation}\label{E1a}
	\int_0^t\int_{\R^d} D_\e^2 (f_\e+f)\, dx\, ds \leq C \int_0^t  \|D_\e\|_{H^1}^2\, ds.
\end{equation} 
\medskip

 \noindent{\bf The term $T_2$.} In order to estimate $T_2$, we first observe that $W_3^1(\R^d)$ embeds continuously in~${L_{12}(\R^d)}$ due to $d\le 4$. 
 We then infer from Gagliardo-Nirenberg's inequality \cite{Ni59} that 
 \begin{equation}\label{GN1}
 \|\nabla D_\e\|_{12}\leq C\|D_\e\|^a_{W^2_3} \|\nabla D_\e\|_2^{1-a}\leq C\|d_\e\|_3^a \|\nabla D_\e\|_2^{1-a},
 \end{equation}
 where     
 \[
 a := \frac{5d}{2(d+6)}\in(0,1]. 
 \]
 H\"older's inequality and \eqref{GN1} now imply that
\begin{equation*}
\begin{aligned}
 \|f_\e\nabla h_\e \cdot \nabla D_\e\|_1&\leq \|f_\e\|_{12/5}\|\nabla h_\e\|_2\|\nabla D_\e\|_{12}\\[1ex]
 &\leq C\|f_\e\|_{12/5}\|\nabla h_\e\|_2\|d_\e\|_3^a\|\nabla D_\e\|_2^{1-a}.
\end{aligned}  
 \end{equation*}
Taking advantage of Young's inequality, we then get
 \begin{align*}
 2\e\|f_\e\nabla h_\e \cdot \nabla D_\e\|_1&\leq \frac{1}{4}\|d_\e\|_3^3+C \left( \e\|f_\e\|_{12/5}\|\nabla h_\e\|_2\| \nabla D_\e\|_2^{1-a} \right)^{\tfrac{3}{3-a}}\\[1ex]
  &= \frac{1}{4}\|d_\e\|_3^3+C \left( \e^{1/2}\|\nabla h_\e\|_2 \|\nabla D_\e\|_2 \right)^{\tfrac{3(1-a)}{3-a}} \left( \e^{1/2} \|\nabla h_\e\|_2 \right)^{\tfrac{3a}{3-a}} \left( \e^{1/2}\|f_\e\|_{12/5} \right)^{\tfrac{3}{3-a}}\\[1ex]
 &\leq \frac{1}{4}\|d_\e\|_3^3+ \e\|\nabla h_\e\|_2^2 \|\nabla D_\e\|_2^2 +C \left( \e^{1/2} \|\nabla h_\e\|_2 \right)^{\tfrac{6a}{3+a}} \left( \e^{1/2}\|f_\e\|_{12/5} \right)^{\tfrac{6}{3+a}}.
 \end{align*}
 Consequently, since $d_\e\le f_\e+f$,
 \begin{equation*}
 	\begin{split}
	T_2 & \leq  \frac{1}{4} \int_0^t \int_{\R^d} (f_\e+f) d_\e^2\, dx\, ds + \e \int_0^t \|\nabla h_\e\|_2^2 \| D_\e\|_{H^1}^2\, ds \\
	& \qquad +C\e^{\tfrac{3 }{3+a}} \int_0^t \left( \e^{1/2}\|\nabla h_\e\|_2 \right)^{\tfrac{6a}{3+a}} \|f_\e\|_{12/5}^{\tfrac{6}{3+a}}\,ds.
	\end{split}
 \end{equation*}
Using H\"older's inequality and \eqref{eq:DED}, we further have
\begin{align*}
\int_0^t \left( \e^{1/2}\|\nabla h_\e\|_2 \right)^{\tfrac{6a}{3+a}} \|f_\e\|_{12/5}^{\tfrac{6}{3+a}} \,ds
&\leq \left( \int_0^t\e\|\nabla h_\e\|_2^2\, ds \right)^{\tfrac{3a}{3+a}} \left( \int_0^t\|f_\e\|_{12/5}^{\tfrac{6}{3-2a}}\,ds \right)^{\tfrac{3-2a}{3+a}}\\[1ex]
&\leq C(1+t)^{\tfrac{3a}{3+a}} \left( \int_0^t\|f_\e\|_{12/5}^{\tfrac{6}{3-2a}}\,ds \right)^{\tfrac{3-2a}{3+a}}.
\end{align*}
Furthermore, by Gagliardo-Nirenberg's inequality and Lemma~\ref{L:1}~(i), 
\[
\|f_\e\|_{12/5}\leq C\|\nabla f_\e\|_2^b\|f_\e\|_2^{1-b}\leq C\|\nabla f_\e\|_2^b \qquad\text{with $b=\frac{d}{12}\in(0,1)$},
\]
which implies, together with  \eqref{eq:DED} and the property $d/(2(3-2a))\leq 2$, that 
 \begin{align*}
\left( \int_0^t\|f_\e\|_{12/5}^{\tfrac{6}{3-2a}}\,ds \right)^{\tfrac{3-2a}{3+a}}\leq C \left( \int_0^t\|\nabla f_\e\|_2^{\tfrac{d}{2(3-2a)}}\,ds \right)^{\tfrac{3-2a}{3+a}}\leq C(1+t)^{\tfrac{3-2a}{3+a}}.
\end{align*}
Recalling the definition of $a$, we conclude that
 \begin{equation}\label{E2}
T_2\leq  \frac{1}{4} \int_0^t \int_{\R^d} (f_\e+f) d_\e^2\, dx\, ds + \e \int_0^t \|\nabla h_\e\|_2^2 \| D_\e\|_{H^1}^2\, ds +C(1+t)\e^{\tfrac{6d+36}{11d+36}}
 \end{equation}
 for all $t>0$.
 
 \medskip
 
 \noindent{\bf Applying Gronwall's lemma.} Gathering \eqref{II5}, \eqref{E1}, \eqref{E1a}, and \eqref{E2}, we arrive at
 \begin{align}\label{Gro}
   \|D_\e(t)\|_{H^1}^2\leq \int_0^t \left( C + \e \|\nabla h_\e\|_2^2 \right)\|D_\e\|_{H^1}^2\, ds+C(1+t)\e^{\tfrac{6d+36}{11d+36}}, \qquad t\geq0.
  \end{align}
  Recalling \eqref{eq:DED}, a direct application of Gronwall's inequality now leads us to the desired estimate~\eqref{CR2}.
  \medskip

 \noindent{\bf Uniqueness of the solution to \eqref{PME}.}
If $f_1$ and $f_2$ are two solutions to \eqref{PME} corresponding to the same initial data $f_0$, then we can perform similar computations as those leading to \eqref{Gro} to obtain that 
    \begin{align*}
   \|(1-\Delta)^{-1}[(f_1-f_2)(t)]\|_{H^1}^2\leq C\int_0^t \|(1-\Delta)^{-1}[(f_1-f_2)]\|_{H^1}^2\, ds, \qquad t\geq0,
  \end{align*}
  and Gronwall's lemma then implies  $f_1=f_2.$
\end{proof}

We complete this section with the proof of Corollary~\ref{C:2}.

\begin{proof}[Proof of Corollary~\ref{C:2}]
Keeping the notation introduced in the proof of Theorem~\ref{MT2}, it readily follows from \eqref{X5} and \eqref{E1} that
\begin{equation}
	T_1  \leq \frac{1}{4}\int_0^t\int_{\R^d} d_\e^2(f_\e+f) \, dx\, ds + \kappa  \int_0^t \|D_\e\|_2^2\, ds. \label{X6}
\end{equation}
Using again \eqref{X5} along with \eqref{eq:DED} and H\"older's and Young's inequality, we find
\begin{equation}
	T_2 \le 2\kappa\e \int_0^t \|\nabla h_\varepsilon\|_2 \|D_\e\|_2\, ds \le \int_0^t \|D_\e\|_{H^1}^2\, ds + C \e (1+t). \label{X7}
\end{equation}
Combining \eqref{II5}, \eqref{X6}, and \eqref{X7} gives
\begin{equation*}
	\|D_\e(t)\|_{H^1}^2 \le (1+\kappa) \int_0^t \|D_\e\|_{H^1}^2\, ds + C \e (1+t), \qquad t\ge 0,
\end{equation*}
and applying Gronwall's lemma completes the proof.
\end{proof}

 \section{The limiting behavior of   $g_\e$}\label{Sec:4}

In this section we establish our last main result stated in Theorem~\ref{MT3}.
Before going on, we point out that all the estimates for the family $(g_\e)_{\e\in (0,1)}$  provided by Theorem~\ref{T:1} involve the function~$g_\e$ multiplied by a positive power of $\e$ when $\alpha>0$, see Lemma~\ref{L:1}, except the conservation of mass~$\|g_\e(t)\|_1=1$, which stems from $g_\e(t)\in \cK$, $t\geq 0$. Nevertheless, exploiting this property, we establish below the convergence of $(g_\e)_{\e\in (0,1)}$ towards the initial condition $g_0$ for $0\le \alpha<1/(d+2)$, without any restriction on the space dimension~$d\geq1$.

\begin{proof}[Proof of Theorem~\ref{MT3}]
To start, we use  \eqref{T1b} and the continuous  embedding~${H^{1+d}(\R^d)\hookrightarrow W^1_\infty(\R^d)}$ to deduce that 
\begin{align*}
  \Big|\int_{\R^d}(g_\e(t)-g_0)\xi\, dx\Big|&\leq \int_0^t\int_{\R^d}|J_{g_\e}\cdot\nabla \xi|\,\, dx\, dt\leq \|J_{g_\e}\|_{L_1((0,t)\times\R^d)}\|\nabla\xi\|_\infty\\[1ex]
  &\leq C\|J_{g_\e}\|_{L_1((0,t)\times\R^d)}\|\xi\|_{H^{1+d}},
 \end{align*}
 for all $t\geq 0$ and $\xi\in C^\infty_c(\R^d)$, where we recall that $ J_{g_\e}=\mu\e g_\e\nabla h_\e$.
 Consequently we have
 $$
\|g_\e(t)-g_0\|_{H^{-1-d}} \leq C\|J_{g_\e}\|_{L_1((0,t)\times\R^d)}\qquad \text{for $ \, t\geq0 $ and $\, \e>0,$}
 $$
 and it remains to estimate the norm $\|J_{g_\e}\|_{L_1((0,t)\times\R^d)}.$

Recalling that $\|g_\e(t)\|_1=1$, $t\geq0$, Gagliardo-Nirenberg's inequality yields
\begin{align*}
\|g_\e(t)\|_2 &\leq C\|\nabla g_\e(t)\|_2^{\tfrac{d}{d+2}}\|g_\e(t)\|_1^{\tfrac{2}{2+d}}\\
&\leq C\|\nabla g_\e(t)\|_2^{\tfrac{d}{d+2}} \\
& \le C \left( \|\nabla f_\e(t)\|_2 + \|\nabla h_\e(t)\|_2 \right)^{\tfrac{d}{d+2}} \\
& \le C \e^{-\tfrac{d}{2(d+2)}} \left( \|\nabla f_\e(t)\|_2 + \sqrt{\e} \|\nabla h_\e(t)\|_2 \right)^{\tfrac{d}{d+2}} , \qquad t\ge 0.
\end{align*}
Taking now advantage of~\eqref{eq:DED},  we get 
$$
\e^{\tfrac{d}{2(d+2)}}\|g_\e\|_{L_{2(d+2)/d}(0,t;L_2(\R^d))}\leq C(1+t)^{\tfrac{d}{2(d+2)}},\qquad t \geq0,
$$
and H\"older's inequality leads us to
$$
\e^{\tfrac{d}{2(d+2)}}\|g_\e\|_{L_2((0,t)\times\R^d)}\leq C(1+t)^{\tfrac{1}{2}},\qquad t \geq0.
$$
Recalling from \eqref{eq:DED} that
$$
\sqrt{\e}\|\nabla h_\e\|_{L_2((0,t)\times\R^d)}\leq C(1+t)^{\tfrac{1}{2}},\qquad t \geq0,
$$
we obtain, by using once more H\"older's inequality and the definition of $J_{g_\e}$, that 
$$
\|J_{g_\e}\|_{L_1((0,t)\times\R^d)}\leq C(1+t)\e^{\tfrac{1}{d+2}-\alpha},\qquad t \geq0,
$$
which proves the claim.
\end{proof}

 \section*{Acknowledgments}

Part of this work was carried out while PhL enjoyed the hospitality and support of DFG Research Training Group~2339 ``Interfaces, Complex Structures, and Singular Limits in Continuum Mechanics - Analysis and Numerics'' at Fakult\"at f\"ur Mathematik, Universit\"at Regensburg.

\bibliographystyle{siam}
\bibliography{PB21}
\end{document}